\pdfoutput=1
\RequirePackage{ifpdf}
\documentclass[pdftex]{sigma}
\usepackage[T1]{fontenc}
\usepackage[utf8]{inputenc}
\usepackage{lmodern}
\usepackage{mathtools,booktabs,array,tabularx}
\usepackage{tikz,pgfplots}
\pgfplotsset{compat=1.18}
\usepackage[section]{placeins}
\usepackage[numbers,sort&compress]{natbib}
\usepackage{enumitem,microtype,xurl}
\definecolor{routeblue}{RGB}{30,78,121}
\definecolor{routeorange}{RGB}{180,83,9}
\pgfplotsset{every axis/.append style={font=\small,tick label style={font=\footnotesize},legend style={font=\footnotesize},grid=major,grid style={gray!15},axis line style={gray!70}}}
\usetikzlibrary{arrows.meta,positioning}
\numberwithin{equation}{section}
\newtheorem{Theorem}{Theorem}[section]
\newtheorem{Proposition}[Theorem]{Proposition}
\newtheorem{Lemma}[Theorem]{Lemma}
\newtheorem{Corollary}[Theorem]{Corollary}
\theoremstyle{definition}\newtheorem{Definition}[Theorem]{Definition}
\theoremstyle{remark}\newtheorem{Remark}[Theorem]{Remark}
\newcommand{\C}{\mathbb C}
\newcommand{\norm}[1]{\left\lVert#1\right\rVert}
\newcommand{\ket}[1]{|#1\rangle}
\newcommand{\opt}{\mathrm{opt}}
\newcommand{\geo}{\mathrm{geo}}
\newcommand{\sgn}{\operatorname{sgn}}

\hypersetup{pdftitle={Weighted flip graphs from cyclic quantum mutations: pentagon inequalities and exponential separations},pdfauthor={Xiaomin Tang}}
\begin{document}
\thispagestyle{empty}
\ShortArticleName{Weighted Flip Graphs from Cyclic Quantum Mutations}
\ArticleName{Weighted Flip Graphs\\ from Cyclic Quantum Mutations:\\ Pentagon Inequalities and Exponential Separations}
\Author{Xiaomin TANG}
\AuthorNameForHeading{X. Tang}
\Address{School of Mathematics and Statistics, Heilongjiang University,\\ Harbin 150080, P.R. China}
\Email{\mail{tangxm@hlju.edu.cn}}
\Abstract{Positive cyclic quantum mutations at a third root of unity induce an explicit weight on polygon flip graphs. We study the resulting discrepancy between minimum flip count and minimum accumulated normalization loss. Every pentagon three-to-two replacement strictly decreases weight, and in connected types $A_1$, $A_2$, and $A_3$ every minimum-weight path is a flip geodesic. In type $A_4$, however, a six-flip path has smaller weight than the unique five-flip geodesic. A compatible positive projective gluing construction amplifies this example to an exponential success-probability separation in every connected type $A_n$, $n\ge4$. We also give a quantitative perturbation bound preserving the separation. The low-rank minimality statement has an exhaustive finite certificate; the pentagon inequality and the all-rank construction are analytic. In the associated sequential postselection model, differences of path weights determine success-probability ratios independently of the input state. The comparison holds for separately normalized exact mutation steps.}
\Classification{13F60; 05C12; 81P45}
\section{Introduction}
The flip graph of a polygon carries two natural optimization problems once its flips are represented by nonunitary operators: minimizing the number of flips and minimizing the normalization loss of their sequential implementation. Polygon flip graphs realize the type-$A$ exchange combinatorics of cluster algebras \cite{FZ,FST}. The first optimization problem depends only on the exchange graph. The second also depends on positive coefficient data. This paper determines the smallest connected type $A$ in which the two objectives disagree for a specific weight induced by cyclic quantum mutations, and constructs an exponential separation in arbitrary rank.

For ordered positive Pl\"ucker coordinates and a quadrilateral $i<j<k<l$, the weight is
\begin{equation}\label{eq:introweight}
w(i,j,k,l)=\log\frac{A^{1/3}+B^{1/3}}{(A+B)^{1/3}},\qquad
A=\lambda_{ij}\lambda_{kl},\quad B=\lambda_{il}\lambda_{jk}.
\end{equation}
It is unchanged by reversing the flip. Our main geometric statement is that every pentagon shortening lowers this weight, but these local improvements do not force global weight minimizers to be flip geodesics. The first obstruction occurs in a seven-gon. Positive projective gluing then places arbitrarily many copies of the obstruction inside one connected exchange type, with globally compatible coefficients.

The weight has an exact operational interpretation. In the positive three-level cyclic representation, it is the logarithm of the squared norm of a flip operator. Normalizing each accepted branch by that norm converts path-weight differences into input-independent ratios of all-success probabilities. The geometric results therefore apply to a specified family of nonunitary quantum operations, while the optimization problem itself is formulated entirely on the weighted flip graph.

\subsection{Related work and the present problem}
The operator identities underlying this question have a substantial history. Faddeev and Kashaev established quantum dilogarithm pentagon identities, including a finite-dimensional realization with parameters on a Fermat curve \cite{FK}. Kashaev used quantum dilogarithms to quantize Teichm\"uller spaces \cite{Kashaev98}, while Fock and Chekhov constructed a mapping-class-equivariant noncommutative deformation of their function algebra \cite{CF}. Fock and Goncharov developed cluster ensembles and their quantization \cite{FGensemble}, and constructed representations of quantum cluster varieties using quantum dilogarithms \cite{FGrepresentation}. These works explain why algebraic changes of coordinates can be represented by consistent operator products.

The connection between mutation periodicity and dilogarithm identities was developed by Kashaev and Nakanishi \cite{KN}. Ip and Yamazaki studied root-of-unity limits and cyclic dilogarithm identities associated with quiver mutation sequences \cite{IY}. Finite-dimensional constructions from quantum-torus representation theory were given by Kim \cite{Kim}. More recently, Cheung, Fr\'ias-Medina and Magee extended quantization to deformed cluster Poisson varieties \cite{CFMM}. The direct representation-theoretic input for this paper is Ishibashi's cyclic quantum Teichm\"uller theory \cite{Ishibashi}, which organizes the cyclic parameters by coefficient mutation and supplies the projective representation of the dotted Ptolemy groupoid that we use. Our norm and path calculations concern this specified positive three-level realization; they do not supply an alternative construction of that representation.

Probabilistic implementation of a nonunitary linear map is also established. Terashima and Ueda formulated nonunitary quantum circuits based on measurement and studied improvements through reversing measurements \cite{Terashima}. Azses, Ruhman and Sela investigated nonunitary gates within measurement-based computation and optimized their success probability \cite{Azses}. Leadbeater, Fitzpatrick, Mu\~noz Ramo and Thom constructed probabilistic nonunitary Trotter circuits for imaginary-time evolution using ancillary measurements \cite{Leadbeater}. These studies motivate the importance of normalization losses and the choice of an implementation. Here we fix each mutation's largest exact success branch and compare different mutation factorizations of the same endpoint map. Recovery after failure and unrestricted synthesis of that map are different optimization problems and are not included in our comparison.

On the geometric side, polygon flip distance is the rotation-distance problem studied by Sleator, Tarjan and Thurston \cite{STT}; Pournin determined the diameter of the associahedron \cite{Pournin}. Common-face geodesic properties were studied by Ceballos and Pilaud \cite{Ceballos}, and Williams gave a uniform non-leaving-face theorem for Coxeter associahedra \cite{Williams}. A separate obstruction to a natural local strategy was identified by Cleary and Maio: successively reducing conflicts between diagonals need not produce an unweighted geodesic \cite{ClearyMaio}. Their objective is flip count and their local criterion is conflict reduction. Our objective is a sum of operator-norm costs, and our comparison is between a weighted optimum and all paths of minimum flip count. The two failures of local intuition should therefore be distinguished.

The new results concern the geometry induced by~\eqref{eq:introweight}. The cyclic representation and the one-step norm-normalization principle are established inputs. Neither input determines which mutation paths minimize total loss. Likewise, the unweighted non-leaving-face theorem constrains shortest paths, but does not compare their weights with longer competitors. We supply that comparison through a strict pentagon inequality, a sharp low-rank boundary, and a compatible all-rank construction. In particular, the exponential statement is obtained within a connected type, rather than by assuming independent coefficient assignments on disconnected components.

\subsection{Implementation model and main results}
Throughout, the coefficients are positive and compatible, and the cyclic dimension is three. We compare paths with the same fixed dotted and labelled endpoints in the representation of \cite{Ishibashi}. Changes of dots and labels are deterministic unitary operations. At each flip, the accepted branch is the largest scalar multiple of the exact mutation operator compatible with a quantum operation. Success means that every flip succeeds. The model does not include recovery after failure, gate duration, noise, or the cost of the auxiliary unitary operations.

A \emph{flip geodesic} is a path with the least number of flips between its endpoints. A \emph{success-optimal path} maximizes the all-success probability in the specified model. For a unit input $\psi$, denote the corresponding maximum over all paths by $P_{\opt}(\psi)$, and the maximum restricted to flip geodesics by $P_{\geo}(\psi)$.

\begin{Theorem}[Local cost]\label{thm:local}
A positive cyclic flip with parameters $a,b>0$ and $a^3+b^3=1$ has implementation cost $\log(a+b)$. Reverse flips have the same cost. Square exchanges preserve cost, whereas every pentagon replacement of three edges by two strictly decreases cost.
\end{Theorem}

\begin{Theorem}[Minimal connected type; computer assisted]\label{thm:minimal}
In connected types $A_1$, $A_2$, and $A_3$, every success-optimal path is a flip geodesic. There exist positive compatible coefficients in type $A_4$ for which no flip geodesic is success optimal. Thus $A_4$ is the smallest connected $A$-type rank admitting this separation in the present model.
\end{Theorem}

\begin{Theorem}[Exponential separation]\label{thm:exponential}
For every $n\ge4$, put $k=\lfloor(n+1)/5\rfloor$. There exist positive compatible coefficients and two seeds of connected type $A_n$ such that, for every unit input $\psi$,
\begin{equation}\label{eq:mainbound}
\frac{P_{\opt}(\psi)}{P_{\geo}(\psi)}>\exp(2k/15).
\end{equation}
Let
\begin{equation}\label{eq:C}
C=\frac{1+\sqrt[3]{3}}{\sqrt[3]{4}}.
\end{equation}
For each fixed $1<c<C$, one may choose a single local parameter, independent of $n$, so that the constructed ratio is greater than $c^k$ at every rank $n\ge4$.
\end{Theorem}

The representation theory and the optimal normalization of one nonunitary gate are inputs to this work. The results concern their induced weighted flip geometry. The lower-rank part of Theorem~\ref{thm:minimal} uses a complete finite combinatorial certificate; the all-rank amplification in Theorem~\ref{thm:exponential} is analytic. Its geometric input is the unweighted non-leaving-face property \cite{STT,Ceballos,Williams}. We do not assume a weighted version of that property.

\section{Polygon seeds and positive coefficients}\label{sec:coefficients}
\subsection{Exchange data}
\begin{Definition}
Let $\mathcal T$ be a triangulation of a convex $m$-gon. Its $m-3$ internal diagonals are mutable edges and its boundary sides are frozen edges. A flip replaces the diagonal of a quadrilateral by the other diagonal. The graph of triangulations joined by flips is denoted by $\mathcal F_m$; it is the exchange graph of connected type $A_{m-3}$.
\end{Definition}

The extended signed adjacency matrix $\varepsilon$ is obtained by summing contributions from triangles. Choose the orientation so that a triangle with ordered vertices $i<j<k$ contributes $+1$ to consecutive entries of the cyclic edge list $(ij,jk,ik)$ and $-1$ in the reverse direction. Its mutable part is skew-symmetric. Under mutation at $\kappa$,
\begin{equation*}
\varepsilon'_{ef}=\begin{cases}
-\varepsilon_{ef},&e=\kappa\text{ or }f=\kappa,\\
\varepsilon_{ef}+[\varepsilon_{e\kappa}]_+[\varepsilon_{\kappa f}]_+
-[-\varepsilon_{e\kappa}]_+[-\varepsilon_{\kappa f}]_+,&\text{otherwise},
\end{cases}
\end{equation*}
where $[x]_+=\max(x,0)$. For context, the algebraic quantum seed construction starts from a based quantum torus with monomials $M(\alpha)$, $\alpha\in\mathbb Z^N$, satisfying
\begin{equation*}
M(\alpha)M(\beta)=v^{\Lambda(\alpha,\beta)}M(\alpha+\beta),
\end{equation*}
where $v$ is an indeterminate and $\Lambda$ is an integral skew-symmetric form. An extended exchange matrix $\widetilde B$ is compatible when
$\widetilde B^{\mathsf T}\Lambda=(D\ \ 0)$ with $D$ a positive diagonal matrix, after mutable indices are placed first. Mutation in direction $\kappa$ replaces the corresponding generator by
\begin{equation*}
M\left(-e_\kappa+\sum_i[b_{i\kappa}]_+e_i\right)
+M\left(-e_\kappa+\sum_i[-b_{i\kappa}]_+e_i\right),
\end{equation*}
where $e_i$ are the standard lattice vectors; the other generators are unchanged \cite{BZ}. The associated skew form and exchange matrix mutate compatibly. This recalls the algebraic meaning of quantum mutation. The operators studied below are the cyclic geometric realization supplied by \cite{Ishibashi}, with the explicit coefficient convention fixed in this section. We do not construct a new quantum seed or identify the auxiliary dotted-triangle coordinates with all generators of the above torus.

\subsection{Positive Pl\"ucker data}
The distinction between ordinary coefficient mutation and quantum cluster Poisson transformations is relevant here. The coefficient framework comes from \cite{FZIV}, while the Poisson-side quantization and its deformation parameters are treated in \cite{FGensemble,CFMM}. Our explicit positive dictionary below is chosen to match the cyclic representation of \cite{Ishibashi}.
Choose homogeneous vectors $u_i\in\mathbb R^2$ with positive ordered determinants
$\lambda_{ij}=\det(u_i,u_j)>0$ for $i<j$. One example is
$u_i=(1,t_i)$ with $t_1<\cdots<t_m$, giving $\lambda_{ij}=t_j-t_i$.
For $i<j<k<l$, the Pl\"ucker relation reads
\begin{equation}\label{eq:plucker}
\lambda_{ik}\lambda_{jl}=\lambda_{ij}\lambda_{kl}+\lambda_{il}\lambda_{jk}.
\end{equation}
For an edge $e$ of a triangulation, including boundary edges when needed, define
\begin{equation*}
Y_e=\prod_f\lambda_f^{\varepsilon_{ef}}.
\end{equation*}
In a quadrilateral whose current diagonal is $\kappa=ik$, put
\begin{equation*}
A=\lambda_{ij}\lambda_{kl},\qquad B=\lambda_{il}\lambda_{jk}.
\end{equation*}
Then $Y_\kappa=A/B$. The new diagonal has coefficient $Y'_\kappa=B/A$.

Figure~\ref{fig:quadrilateral} fixes the geometric convention and the two positive products entering the weight.
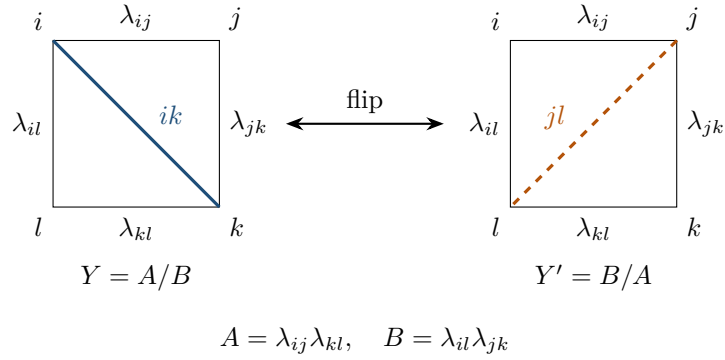
\begin{figure}[tbp]
\centering
\begin{tikzpicture}[scale=1.1,font=\small]
\foreach \shift/\tag in {0/a,5.5/b}{
\begin{scope}[xshift=\shift cm]
\coordinate (\tag i) at (-1,1);\coordinate (\tag j) at (1,1);
\coordinate (\tag k) at (1,-1);\coordinate (\tag l) at (-1,-1);
\draw (\tag i)--node[above]{$\lambda_{ij}$}(\tag j)--node[right]{$\lambda_{jk}$}(\tag k)--node[below]{$\lambda_{kl}$}(\tag l)--node[left]{$\lambda_{il}$}cycle;
\node[above left] at (\tag i){$i$};\node[above right] at (\tag j){$j$};
\node[below right] at (\tag k){$k$};\node[below left] at (\tag l){$l$};
\end{scope}}
\draw[routeblue,very thick] (ai)--node[above right,pos=.57]{$ik$}(ak);
\draw[routeorange,very thick,dashed] (bj)--node[above left,pos=.6]{$jl$}(bl);
\draw[<->,>=Stealth,thick] (1.8,0)--node[above]{flip}(3.7,0);
\node at (0,-1.85){$Y=A/B$};\node at (5.5,-1.85){$Y'=B/A$};
\node at (2.75,-2.6){$A=\lambda_{ij}\lambda_{kl},\quad B=\lambda_{il}\lambda_{jk}$};
\end{tikzpicture}
\caption{A quadrilateral flip. The two orientations have reciprocal coefficients and the same weight $w=\log g(A,B)$. The boundary labels are unchanged.}\label{fig:quadrilateral}
\end{figure}

\begin{Lemma}[Coefficient compatibility]\label{lem:coeff}
The positive coefficients satisfy
\begin{equation}\label{eq:Ymut}
Y'_\kappa=Y_\kappa^{-1},\qquad
Y'_e=Y_e\bigl(1+Y_\kappa^{-\sgn(\varepsilon_{e\kappa})}\bigr)^{-\varepsilon_{e\kappa}}
\quad(e\ne\kappa),
\end{equation}
with $Y'_e=Y_e$ if $\varepsilon_{e\kappa}=0$. The positive parameters
\begin{equation}\label{eq:ab}
a_e=(1+Y_e)^{-1/3},\qquad b_e=Y_e^{1/3}(1+Y_e)^{-1/3}
\end{equation}
satisfy the cyclic coefficient mutation rule, with $(a'_\kappa,b'_\kappa)=(b_\kappa,a_\kappa)$.
\end{Lemma}
\begin{proof}
Contributions from triangles outside the quadrilateral cancel in the coefficient ratios. Equation~\eqref{eq:plucker} gives
\begin{align*}
\frac{Y'_{ij}}{Y_{ij}}=\frac{Y'_{kl}}{Y_{kl}}&=\frac{A+B}{B}=1+Y_\kappa,\\
\frac{Y'_{jk}}{Y_{jk}}=\frac{Y'_{il}}{Y_{il}}&=\frac{A}{A+B}=\frac{Y_\kappa}{1+Y_\kappa}.
\end{align*}
The corresponding signed adjacency entries are $-1$ and $+1$, respectively. This proves~\eqref{eq:Ymut}. Put $y_e=b_e/a_e=Y_e^{1/3}$. Taking the unique positive cube root yields
\begin{equation*}
y'_e=\begin{cases}
y_e b_\kappa^{\varepsilon_{e\kappa}},&\varepsilon_{e\kappa}\ge0,\\
y_e a_\kappa^{\varepsilon_{e\kappa}},&\varepsilon_{e\kappa}<0.
\end{cases}
\end{equation*}
This is the required positive coefficient rule. No root-of-unity choice is introduced by taking positive roots.
\end{proof}

Because the Pl\"ucker data define $Y_e$ in every triangulation, the positive pairs in~\eqref{eq:ab} form a global coefficient collection in the sense of \cite[Definition 4.5]{Ishibashi}; compatibility holds on every flip by Lemma~\ref{lem:coeff}.

Our parameter order is $(a,b)=(p^-,p^+)$ in the cyclic notation. Reversing the orientation requires changing both the signed adjacency convention and the coefficient dictionary. The symmetry of the eventual edge cost under $Y\mapsto Y^{-1}$ does not remove this requirement at the representation level.

\section{Cyclic operators and postselected implementation}\label{sec:operators}
\subsection{The three-level flip}
Let $\omega=\exp(2\pi\mathrm i/3)$ and $q=\omega^2$, so $q^2=\omega$ and $q^3=1$. On $\C^3$, indexed by $\mathbb Z/3\mathbb Z$, define
\begin{equation*}
X\ket{j}=\ket{j+1},\qquad Z\ket{j}=\omega^j\ket{j},\qquad ZX=\omega XZ.
\end{equation*}
On two tensor factors set
\begin{equation*}
K=q^{-1}X_1^{-1}Z_1X_2,\qquad S\ket{i,j}=\ket{i-j,j}.
\end{equation*}
The operator $K$ is unitary, $K^3=I$, and its eigenvalues $1,\omega,\omega^2$ each have multiplicity three. Let $E_j$ be the spectral projection at $\omega^j$. For $a,b>0$ with $a^3+b^3=1$, define
\begin{align}
d_0&=1,&d_1&=a+b\omega^2,&d_2&=(a+b)d_1,\notag\\
D(K)&=\sum_{j=0}^2d_jE_j,&T(a,b)&=D(K)S.\label{eq:T}
\end{align}
The cyclic recursion is $d_{j+1}=(a+bq\omega^j)d_j$, with $d_3=d_0$. Its closure follows from
\begin{equation*}
(a+b\omega^2)(a+b)(a+b\omega)=a^3+b^3=1.
\end{equation*}

\begin{Proposition}[Norm and determinant]\label{prop:norm}
The singular values of $T(a,b)$ are
\begin{equation*}
1,\qquad (a+b)^{-1/2},\qquad(a+b)^{1/2},
\end{equation*}
each with multiplicity three. In particular,
\begin{equation}\label{eq:norm}
\norm{T}^2=\norm{T^{-1}}^2=a+b,\qquad |\det T|=1.
\end{equation}
No nonzero scalar multiple of $T$ is unitary.
\end{Proposition}
\begin{proof}
Since $S$ is unitary, the singular values are the moduli of the coefficients $d_j$. Moreover,
\begin{equation*}
|d_1|^2=a^2-ab+b^2=\frac1{a+b},\qquad |d_2|^2=a+b.
\end{equation*}
Their product is one. Positivity and $a^3+b^3=1$ give $a+b>1$, so the singular values are unequal.
\end{proof}

The dependence of these singular values on the positive coefficient is shown in figure~\ref{fig:spectrum}. Their inequality explains why an exact flip requires postselection.
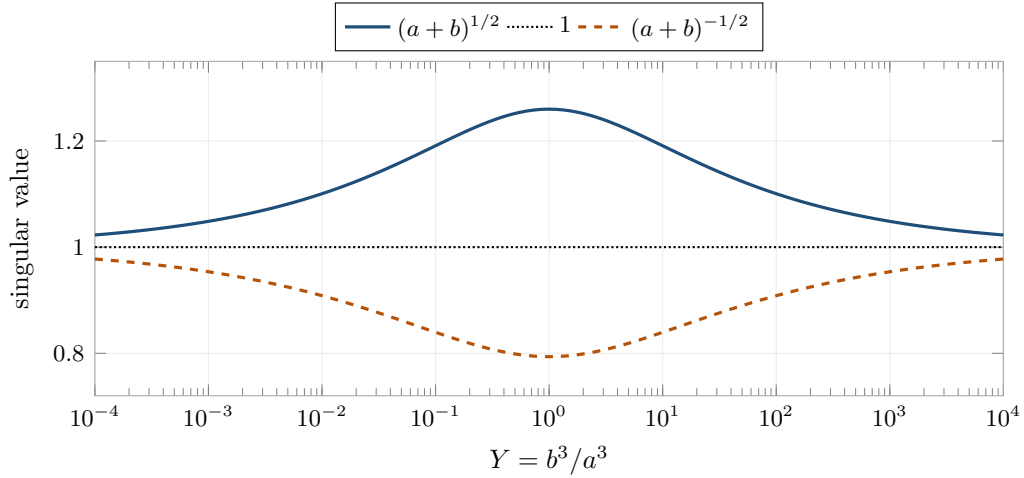
\begin{figure}[tbp]
\centering
\begin{tikzpicture}
\begin{semilogxaxis}[width=.85\textwidth,height=6cm,xlabel={$Y=b^3/a^3$},ylabel={singular value},xmin=1e-4,xmax=1e4,ymin=.72,ymax=1.35,legend style={at={(.5,1.03)},anchor=south,legend columns=3}]
\addplot[routeblue,very thick,domain=-4:4,samples=121] ({10^x},{sqrt((1+(10^x)^(1/3))/(1+10^x)^(1/3))});
\addlegendentry{$(a+b)^{1/2}$}
\addplot[black,densely dotted,thick,domain=-4:4] ({10^x},1);
\addlegendentry{$1$}
\addplot[routeorange,dashed,very thick,domain=-4:4,samples=121] ({10^x},{1/sqrt((1+(10^x)^(1/3))/(1+10^x)^(1/3))});
\addlegendentry{$(a+b)^{-1/2}$}
\end{semilogxaxis}
\end{tikzpicture}
\caption{The three singular values of the flip operator, each of multiplicity three. The curves plot the exact formulas in Proposition~\ref{prop:norm}. The operator norm and inverse norm agree, and the spectrum is symmetric under $Y\mapsto Y^{-1}$.}\label{fig:spectrum}
\end{figure}

\subsection{Optimal exact normalization of one step}
A success Kraus operator $F$ must satisfy $F^\dagger F\le I$. Conditional on success, its output on a unit vector $\psi$ is $F\psi/\norm{F\psi}$, and the success probability is $\norm{F\psi}^2$.

\begin{Lemma}\label{lem:kraus}
For an invertible target $T$, the maximal squared amplitude of an exact scalar success branch is $1/\norm{T}^2$. The same bound applies to the total squared amplitude of multiple corrected accepted branches that realize the same output ray for every input.
\end{Lemma}
\begin{proof}
For $F=cT$, the Kraus inequality is equivalent to $|c|^2\le1/\norm{T}^2$. Equality is realized by $F=T/\norm{T}$ together with the failure operator $(I-T^\dagger T/\norm{T}^2)^{1/2}$.
For a corrected accepted branch $K_\alpha$, exactness means that $T^{-1}K_\alpha\psi$ is proportional to $\psi$ for every vector. A linear operator with this property is scalar: apply it to basis vectors and their pairwise sums. Hence $K_\alpha=c_\alpha T$. Completeness gives
$\sum_\alpha|c_\alpha|^2T^\dagger T\le I$, which proves the assertion.
\end{proof}

This is the familiar norm-normalization principle \cite{Terashima}, specialized here to a mutation operator. It does not require a particular physical dilation.

\subsection{From projective equivalence to path probabilities}
A dotted triangulation consists of a triangulation, a distinguished corner of each triangle, and labels on the triangles. In the cyclic representation, one qutrit is assigned to each triangle, so the state space for an $m$-gon is $(\C^3)^{\otimes(m-2)}$. Ordinary flips lift to the dotted setting after dot rotations and label changes. Those adjustments are unitary and incur no success loss in our model. Operator products are taken in the representation's composition convention, with the rightmost factor acting first.

For completeness, the dictionary with \cite[Section 4.4]{Ishibashi} is $U=Z$ and $P=X$. Its finite Fourier sum gives
\begin{equation*}
\frac13\sum_{r,s=0}^2\omega^{rs}Z_2^rX_1^s\ket{i,j}
=\ket{i-j,j},
\end{equation*}
so its operator $S_{12}$ is exactly our $S$. Its cyclic difference equation gives
\begin{equation*}
D(K)=\Psi_{(b,a)}(K)/\Psi_{(b,a)}(1);
\end{equation*}
see \cite[Lemmas 2.5 and 2.7]{Ishibashi}. Up to a phase, its dot rotation has matrix entries
\begin{equation*}
A_{mk}=3^{-1/2}q^{-m^2}\omega^{-km},\qquad m,k\in\mathbb Z/3\mathbb Z.
\end{equation*}
This is a diagonal unitary times a normalized Fourier matrix. Thus dot rotations and label permutations are unitary in the standard inner product used here. All recurrence factors are nonzero for positive $a,b$; the normalized finite formulas extend the projective identities across any auxiliary branch cuts in the scalar normalization of $\Psi$.

We use the projective representation theorem of \cite[Theorem 4.11]{Ishibashi}. The normalization $d_0=1$ in~\eqref{eq:T} rescales the cyclic dilogarithm by a nonzero scalar and preserves projectivity. Therefore, for two lifted paths with the same endpoints and compatible coefficients,
\begin{equation}\label{eq:projective}
M_\gamma=cM_\eta,\qquad c\in\C^*.
\end{equation}
Here $M_\gamma$ is the unnormalized product, including the unitary adjustments. Tensor embedding preserves the modulus-one determinant property in~\eqref{eq:norm}. Consequently, if the full state-space dimension is $d$, then
\begin{equation*}
1=\frac{|\det M_\gamma|}{|\det M_\eta|}=|c|^d,
\end{equation*}
so $|c|=1$. Projective equivalence is thus sufficient to compare success probabilities.

\begin{Proposition}[Path cost]\label{prop:path}
Define
\begin{equation*}
W_\gamma=\sum_{e\in\gamma}\log\norm{T_e}^2.
\end{equation*}
The all-success branch and its probability are
\begin{equation*}
F_\gamma=e^{-W_\gamma/2}M_\gamma,\qquad
P_\gamma(\psi)=e^{-W_\gamma}\norm{M_\gamma\psi}^2.
\end{equation*}
For two paths with fixed common lifted endpoints,
\begin{equation*}
\frac{P_\gamma(\psi)}{P_\eta(\psi)}=e^{W_\eta-W_\gamma}.
\end{equation*}
Thus a path maximizes success for one nonzero input if and only if it does so for every input.
\end{Proposition}
\begin{proof}
Multiply the scalar normalizations from Lemma~\ref{lem:kraus} and use~\eqref{eq:projective} with $|c|=1$. All operators are invertible, so the probabilities in the ratio are nonzero.
\end{proof}

For comparison, unrestricted implementation of the endpoint map $M$ can use $M/\norm{M}$. Submultiplicativity gives
\begin{equation*}
W_\gamma\ge\log\norm{M}^2.
\end{equation*}
Our path optimization need not attain this lower bound. In particular, the exponential separation below is between sequential mutation paths; it is not an advantage over arbitrary circuit synthesis.

\section{The weighted flip graph and pentagon shortening}\label{sec:pentagon}
Equations~\eqref{eq:ab} and~\eqref{eq:norm} give the edge factor and cost
\begin{equation}\label{eq:g}
g(A,B)=\frac{A^{1/3}+B^{1/3}}{(A+B)^{1/3}},\qquad w=\log g(A,B).
\end{equation}
The factor is symmetric and homogeneous of degree zero. Reverse flips exchange $A$ and $B$, so the graph is undirected as a weighted graph. Independent flips preserve each other's quadrilateral data; hence opposite two-edge arcs of a square have equal cost. Since $g(A,B)>1$, deleting a nontrivial loop strictly decreases cost. On the finite graph $\mathcal F_m$, a weighted minimum therefore exists and is simple.

Figure~\ref{fig:localmoves} contrasts the two local relations used in the reduction argument.
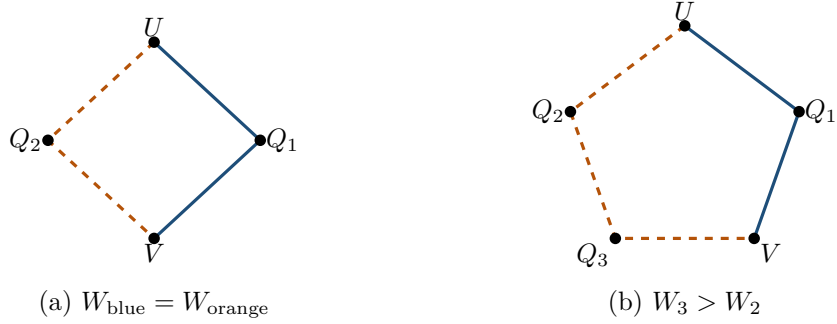
\begin{figure}[tbp]
\centering
\begin{tikzpicture}[font=\small,scale=1.08,every node/.style={inner sep=2pt}]
\begin{scope}
\coordinate (s) at (0,1.2);\coordinate (a) at (1.3,0);\coordinate (t) at (0,-1.2);\coordinate (b) at (-1.3,0);
\draw[routeblue,very thick] (s)--(a)--(t);
\draw[routeorange,very thick,dashed] (s)--(b)--(t);
\foreach \v in {s,a,t,b}{\fill (\v) circle(2pt);}
\node[above] at (s){$U$};\node[below] at (t){$V$};
\node[right] at (a){$Q_1$};\node[left] at (b){$Q_2$};
\node at (0,-2){(a) $W_{\mathrm{blue}}=W_{\mathrm{orange}}$};
\end{scope}
\begin{scope}[xshift=6.5cm]
\coordinate (s) at (0,1.4);\coordinate (a) at (1.4,.35);\coordinate (t) at (.85,-1.2);
\coordinate (b) at (-.85,-1.2);\coordinate (c) at (-1.4,.35);
\draw[routeblue,very thick] (s)--(a)--(t);
\draw[routeorange,very thick,dashed] (s)--(c)--(b)--(t);
\foreach \v in {s,a,t,b,c}{\fill (\v) circle(2pt);}
\node[above] at (s){$U$};\node[below right] at (t){$V$};
\node[right] at (a){$Q_1$};\node[left] at (c){$Q_2$};\node[below left] at (b){$Q_3$};
\node at (0,-2){(b) $W_3>W_2$};
\end{scope}
\end{tikzpicture}
\caption{Local relations in the flip graph. Vertices denote triangulations. (a) Independent flips give two equal-cost square arcs. (b) The dashed three-edge pentagon arc costs strictly more than the solid two-edge arc. These local comparisons do not imply that every globally optimal path minimizes flip count.}\label{fig:localmoves}
\end{figure}

\begin{Proposition}[Strict pentagon inequality]\label{prop:pentagon}
Every three-edge arc of a pentagonal face has strictly greater cost than its complementary two-edge arc.
\end{Proposition}
\begin{proof}
Write $p=(a,b)$ and $r=(c,d)$, where $a^3+b^3=c^3+d^3=1$. Set
\begin{align*}
e&=(a^3+b^3d^3)^{1/3}=(d^3+a^3c^3)^{1/3},\\
p'&=(a/e,bd/e),\qquad r'=(ac/e,d/e),\qquad p''=(e,bc).
\end{align*}
All three pairs are positive and Fermat-normalized. The pentagon relation is
\begin{equation*}
T_{12}(r)T_{13}(p')T_{23}(p'')\doteq T_{23}(p)T_{12}(r'),
\end{equation*}
where $\doteq$ denotes equality up to a scalar. Their success ratio is
\begin{equation}\label{eq:pentagonratio}
\frac{P_3}{P_2}=\frac{(a+b)(ac+d)}{(c+d)(a+bd)(e+bc)}.
\end{equation}
Put
\begin{align*}
M&=(c+d)(a+bd),&V&=(a+b)(ac+d),\\
Q_1&=a^2+b^2-1,&Q_2&=c^2+d^2-1.
\end{align*}
Since $0<a,b,c,d<1$, both $Q_i$ are nonnegative. Also $e>\max(a,d)$.

If $a\ge d$, then $c\ge b$, and polynomial expansion gives
\begin{align*}
M(a+bc)-V={}&abQ_2+d[a(1-a/2)+bc]Q_1\notag\\
&+b(1-c)(a-d)+abd(1-a/2)(2c-b)\notag\\
&+\tfrac12a^2d(1-a)^2+b^2cd(c+d-b)\ge0.
\end{align*}
If $d\ge a$, then $b\ge c$, and
\begin{align*}
M(d+bc)-V={}&[ac+d(1-d)/2+bc(1-c)](d-a)\notag\\
&+\tfrac d2[a+2b+b^2+bc]Q_2+\tfrac{d^2(1-d)}2Q_1\notag\\
&+\tfrac d2[ac+b(1-d)+bc^2](b-c)\notag\\
&+\tfrac{ad^2(1-a)(1-d)}2\notag\\
&+\tfrac{bcd}2[a+2bd+d(1-d)]\ge0.
\end{align*}
Every summand is nonnegative in its respective case. Since $M>0$ and $e>\max(a,d)$, either identity implies $M(e+bc)>V$. Thus~\eqref{eq:pentagonratio} is strictly less than one. Cyclic relabelling treats each three-edge arc.
\end{proof}

This proves Theorem~\ref{thm:local}. The two identities are exact polynomial certificates; their rational coefficients may be checked by expansion without numerical approximation.

\section{Low-rank reduction and a seven-gon separation}\label{sec:separation}
\subsection{A finite reduction lemma}
\begin{Lemma}[Computer-assisted finite lemma]\label{lem:finite}
For $m=4,5,6$, every nongeodesic simple path in $\mathcal F_m$ reduces to a geodesic by square exchanges, pentagon three-to-two replacements, and loop deletion.
\end{Lemma}
\begin{proof}[Finite certificate and verification]
Generate each triangulation as a set of $m-3$ pairwise noncrossing diagonals and join two states if their symmetric difference has cardinality two. Obtain two-dimensional faces by fixing $m-5$ diagonals: the resulting induced graphs are squares or pentagons. For $m=4$ no such faces are needed. Enumerate all simple paths between distinct unordered endpoint pairs, retaining one orientation per pair, and compute endpoint distances by breadth-first search.

Process paths in increasing length. Within each length, partition them into connected classes under square exchanges that remain simple. A nongeodesic class has an exit if one of its paths admits a pentagon shortening, or a square exchange producing a loop that can be deleted. Verify that every nongeodesic class has such an exit to a shorter, previously certified path. In each class, direct a spanning tree towards a chosen exit. The certificate rank of a path is its length followed by its distance to that exit. Every certified move strictly decreases this lexicographic rank. Terminal paths have length equal to their endpoint distance.

The complete counts are given in table~\ref{tab:finite}. For $m=6$, there are three square faces and six pentagonal faces; the greatest square distance in the selected witnesses is three. Every replacement is checked against the corresponding geometric face, and its endpoints and adjacency are verified. The executable verifier described in Appendix~\ref{app:certificate} reproduces this exhaustive check. Thus all paths in the finite domain are covered, and the rank proves termination.
\end{proof}

\begin{table}[tb]
\centering
\caption{Complete simple-path enumeration, with one direction for each unordered endpoint pair.}\label{tab:finite}
\begin{tabular}{rrrrr}\toprule
$m$ & States & Simple paths & Geodesics & Nongeodesics\\\midrule
4 & 2 & 1 & 1 & 0\\
5 & 5 & 20 & 10 & 10\\
6 & 14 & 8889 & 123 & 8766\\\bottomrule
\end{tabular}
\end{table}

The small cases clarify the role of the certificate. For $m=4$, the graph is a single edge, so every simple path is geodesic. For $m=5$, it is a pentagon. A nongeodesic simple path has either three or four edges. A three-edge path shortens directly to its complementary two-edge arc. Replacing the first three edges of a four-edge path by the complementary two-edge arc produces a backtrack, whose deletion leaves the one-edge geodesic. Thus only the hexagon case requires the exhaustive square-class calculation above.

The finite check is independent of the positive coefficient values. It establishes a combinatorial reduction for every path; Proposition~\ref{prop:pentagon} then makes each shortening strictly cost-decreasing for every admissible coefficient collection. The computation consequently proves a uniform low-rank statement, rather than a statement about a sampled set of weights.

Square exchanges preserve cost, while every actual shortening strictly decreases it by Proposition~\ref{prop:pentagon} and positivity. A nongeodesic simple path is therefore more costly than some geodesic. A nonsimple path cannot minimize positive weight. This proves the low-rank assertion of Theorem~\ref{thm:minimal}. It does not assert that all geodesics have equal weight.

\subsection{A human-readable uniqueness proof}
Let the vertices of a seven-gon be $1,\ldots,7$, and abbreviate a diagonal $(i,j)$ by $ij$. Consider
\begin{equation}\label{eq:endpoints}
\mathcal S=\{25,26,27,35\},\qquad \mathcal U=\{13,14,46,47\}.
\end{equation}
Figure~\ref{fig:endpoints} displays these triangulations.

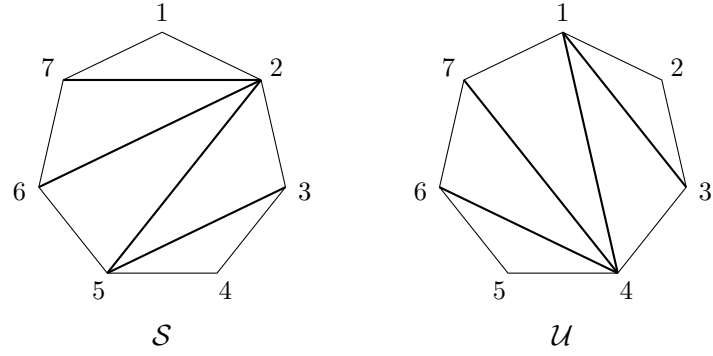
\begin{figure}[tb]
\centering
\begin{tikzpicture}[scale=0.93]
\foreach \shift/\name in {0/S,5.7/U}{
\begin{scope}[xshift=\shift cm]
\foreach \i in {1,...,7}{\coordinate (\name\i) at ({90-360*(\i-1)/7}:1.8);\node at ({90-360*(\i-1)/7}:2.08){\small $\i$};}
\draw (\name1)--(\name2)--(\name3)--(\name4)--(\name5)--(\name6)--(\name7)--cycle;
\end{scope}}
\foreach \a/\b in {2/5,2/6,2/7,3/5}{\draw[thick] (S\a)--(S\b);}
\foreach \a/\b in {1/3,1/4,4/6,4/7}{\draw[thick] (U\a)--(U\b);}
\node at (0,-2.5){$\mathcal S$};\node at (5.7,-2.5){$\mathcal U$};
\end{tikzpicture}
\caption{The seven-gon endpoints. No diagonal is shared. The unique geodesic has five flips, while a six-flip path has lower implementation cost for the coefficients used below.}\label{fig:endpoints}
\end{figure}

\begin{Lemma}\label{lem:unique}
The distance from $\mathcal S$ to $\mathcal U$ is five, and the unique geodesic is
\begin{equation}\label{eq:five}
35\to24,\quad25\to46,\quad26\to47,\quad27\to14,\quad24\to13.
\end{equation}
\end{Lemma}
\begin{proof}
There are four target diagonals, none present initially. The possible first flips replace $25,26,27,35$ by $36,57,16,24$, respectively. None inserts a target diagonal. At least five flips are therefore required, and~\eqref{eq:five} shows that five suffice.

In a five-flip path, every step after the first must increase the number of target diagonals by exactly one. If the first flip inserts $57$ or $16$, no next flip inserts a target diagonal. If the first inserts $36$, the only target-inserting next flip is $35\to46$. The resulting state $\{26,27,36,46\}$ has no target-inserting flip: its new diagonals would be $37,16,24,35$. These first three possibilities are excluded.

It follows that the first flip is $35\to24$. From $\{24,25,26,27\}$ the only target-inserting flip is $25\to46$. The same inspection in the next states forces $26\to47$, then $27\to14$, then $24\to13$. This proves uniqueness.
\end{proof}

\subsection{Exact success comparison}
Assign affine coordinates
\begin{equation*}
t=(0,1,2,3,100,101,200)
\end{equation*}
and take $\lambda_{ij}=t_j-t_i$. Besides~\eqref{eq:five}, use the legal path
\begin{equation*}
\begin{split}
25\to36,\quad35\to46,\quad26\to37,\\
36\to47,\quad27\to13,\quad37\to14.
\end{split}
\end{equation*}
Figure~\ref{fig:paths} displays the two complete paths, including every intermediate triangulation.
\begin{figure}[tbp]
\centering
\begin{tikzpicture}[font=\scriptsize]
\node[anchor=west,font=\small\bfseries] at (-.8,1.45){Five flips};
\begin{scope}[shift={(1.04,0.0)}]
\coordinate (p0n0v1) at (90.0:.62);
\node[font=\tiny,inner sep=0pt] at (90.0:.79){$1$};
\coordinate (p0n0v2) at (38.57142857142857:.62);
\node[font=\tiny,inner sep=0pt] at (38.57142857142857:.79){$2$};
\coordinate (p0n0v3) at (-12.857142857142861:.62);
\node[font=\tiny,inner sep=0pt] at (-12.857142857142861:.79){$3$};
\coordinate (p0n0v4) at (-64.28571428571428:.62);
\node[font=\tiny,inner sep=0pt] at (-64.28571428571428:.79){$4$};
\coordinate (p0n0v5) at (-115.71428571428572:.62);
\node[font=\tiny,inner sep=0pt] at (-115.71428571428572:.79){$5$};
\coordinate (p0n0v6) at (-167.14285714285717:.62);
\node[font=\tiny,inner sep=0pt] at (-167.14285714285717:.79){$6$};
\coordinate (p0n0v7) at (-218.57142857142856:.62);
\node[font=\tiny,inner sep=0pt] at (-218.57142857142856:.79){$7$};
\draw[gray] (p0n0v1)--(p0n0v2)--(p0n0v3)--(p0n0v4)--(p0n0v5)--(p0n0v6)--(p0n0v7)--cycle;
\draw[routeblue,thin] (p0n0v2)--(p0n0v5);
\draw[routeblue,thin] (p0n0v2)--(p0n0v6);
\draw[routeblue,thin] (p0n0v2)--(p0n0v7);
\draw[routeblue,thin] (p0n0v3)--(p0n0v5);
\node at (0,-1.12){$S_0$};
\end{scope}
\draw[-{Stealth},gray] (1.88,0.0)--(2.2800000000000002,0.0);
\begin{scope}[shift={(3.12,0.0)}]
\coordinate (p0n1v1) at (90.0:.62);
\node[font=\tiny,inner sep=0pt] at (90.0:.79){$1$};
\coordinate (p0n1v2) at (38.57142857142857:.62);
\node[font=\tiny,inner sep=0pt] at (38.57142857142857:.79){$2$};
\coordinate (p0n1v3) at (-12.857142857142861:.62);
\node[font=\tiny,inner sep=0pt] at (-12.857142857142861:.79){$3$};
\coordinate (p0n1v4) at (-64.28571428571428:.62);
\node[font=\tiny,inner sep=0pt] at (-64.28571428571428:.79){$4$};
\coordinate (p0n1v5) at (-115.71428571428572:.62);
\node[font=\tiny,inner sep=0pt] at (-115.71428571428572:.79){$5$};
\coordinate (p0n1v6) at (-167.14285714285717:.62);
\node[font=\tiny,inner sep=0pt] at (-167.14285714285717:.79){$6$};
\coordinate (p0n1v7) at (-218.57142857142856:.62);
\node[font=\tiny,inner sep=0pt] at (-218.57142857142856:.79){$7$};
\draw[gray] (p0n1v1)--(p0n1v2)--(p0n1v3)--(p0n1v4)--(p0n1v5)--(p0n1v6)--(p0n1v7)--cycle;
\draw[routeorange,very thick] (p0n1v2)--(p0n1v4);
\draw[routeblue,thin] (p0n1v2)--(p0n1v5);
\draw[routeblue,thin] (p0n1v2)--(p0n1v6);
\draw[routeblue,thin] (p0n1v2)--(p0n1v7);
\node at (0,-1.12){$S_1$};
\end{scope}
\draw[-{Stealth},gray] (3.96,0.0)--(4.36,0.0);
\begin{scope}[shift={(5.2,0.0)}]
\coordinate (p0n2v1) at (90.0:.62);
\node[font=\tiny,inner sep=0pt] at (90.0:.79){$1$};
\coordinate (p0n2v2) at (38.57142857142857:.62);
\node[font=\tiny,inner sep=0pt] at (38.57142857142857:.79){$2$};
\coordinate (p0n2v3) at (-12.857142857142861:.62);
\node[font=\tiny,inner sep=0pt] at (-12.857142857142861:.79){$3$};
\coordinate (p0n2v4) at (-64.28571428571428:.62);
\node[font=\tiny,inner sep=0pt] at (-64.28571428571428:.79){$4$};
\coordinate (p0n2v5) at (-115.71428571428572:.62);
\node[font=\tiny,inner sep=0pt] at (-115.71428571428572:.79){$5$};
\coordinate (p0n2v6) at (-167.14285714285717:.62);
\node[font=\tiny,inner sep=0pt] at (-167.14285714285717:.79){$6$};
\coordinate (p0n2v7) at (-218.57142857142856:.62);
\node[font=\tiny,inner sep=0pt] at (-218.57142857142856:.79){$7$};
\draw[gray] (p0n2v1)--(p0n2v2)--(p0n2v3)--(p0n2v4)--(p0n2v5)--(p0n2v6)--(p0n2v7)--cycle;
\draw[routeblue,thin] (p0n2v2)--(p0n2v4);
\draw[routeblue,thin] (p0n2v2)--(p0n2v6);
\draw[routeblue,thin] (p0n2v2)--(p0n2v7);
\draw[routeorange,very thick] (p0n2v4)--(p0n2v6);
\node at (0,-1.12){$S_2$};
\end{scope}
\draw[-{Stealth},gray] (6.04,0.0)--(6.44,0.0);
\begin{scope}[shift={(7.28,0.0)}]
\coordinate (p0n3v1) at (90.0:.62);
\node[font=\tiny,inner sep=0pt] at (90.0:.79){$1$};
\coordinate (p0n3v2) at (38.57142857142857:.62);
\node[font=\tiny,inner sep=0pt] at (38.57142857142857:.79){$2$};
\coordinate (p0n3v3) at (-12.857142857142861:.62);
\node[font=\tiny,inner sep=0pt] at (-12.857142857142861:.79){$3$};
\coordinate (p0n3v4) at (-64.28571428571428:.62);
\node[font=\tiny,inner sep=0pt] at (-64.28571428571428:.79){$4$};
\coordinate (p0n3v5) at (-115.71428571428572:.62);
\node[font=\tiny,inner sep=0pt] at (-115.71428571428572:.79){$5$};
\coordinate (p0n3v6) at (-167.14285714285717:.62);
\node[font=\tiny,inner sep=0pt] at (-167.14285714285717:.79){$6$};
\coordinate (p0n3v7) at (-218.57142857142856:.62);
\node[font=\tiny,inner sep=0pt] at (-218.57142857142856:.79){$7$};
\draw[gray] (p0n3v1)--(p0n3v2)--(p0n3v3)--(p0n3v4)--(p0n3v5)--(p0n3v6)--(p0n3v7)--cycle;
\draw[routeblue,thin] (p0n3v2)--(p0n3v4);
\draw[routeblue,thin] (p0n3v2)--(p0n3v7);
\draw[routeblue,thin] (p0n3v4)--(p0n3v6);
\draw[routeorange,very thick] (p0n3v4)--(p0n3v7);
\node at (0,-1.12){$S_3$};
\end{scope}
\draw[-{Stealth},gray] (8.120000000000001,0.0)--(8.52,0.0);
\begin{scope}[shift={(9.36,0.0)}]
\coordinate (p0n4v1) at (90.0:.62);
\node[font=\tiny,inner sep=0pt] at (90.0:.79){$1$};
\coordinate (p0n4v2) at (38.57142857142857:.62);
\node[font=\tiny,inner sep=0pt] at (38.57142857142857:.79){$2$};
\coordinate (p0n4v3) at (-12.857142857142861:.62);
\node[font=\tiny,inner sep=0pt] at (-12.857142857142861:.79){$3$};
\coordinate (p0n4v4) at (-64.28571428571428:.62);
\node[font=\tiny,inner sep=0pt] at (-64.28571428571428:.79){$4$};
\coordinate (p0n4v5) at (-115.71428571428572:.62);
\node[font=\tiny,inner sep=0pt] at (-115.71428571428572:.79){$5$};
\coordinate (p0n4v6) at (-167.14285714285717:.62);
\node[font=\tiny,inner sep=0pt] at (-167.14285714285717:.79){$6$};
\coordinate (p0n4v7) at (-218.57142857142856:.62);
\node[font=\tiny,inner sep=0pt] at (-218.57142857142856:.79){$7$};
\draw[gray] (p0n4v1)--(p0n4v2)--(p0n4v3)--(p0n4v4)--(p0n4v5)--(p0n4v6)--(p0n4v7)--cycle;
\draw[routeorange,very thick] (p0n4v1)--(p0n4v4);
\draw[routeblue,thin] (p0n4v2)--(p0n4v4);
\draw[routeblue,thin] (p0n4v4)--(p0n4v6);
\draw[routeblue,thin] (p0n4v4)--(p0n4v7);
\node at (0,-1.12){$S_4$};
\end{scope}
\draw[-{Stealth},gray] (10.2,0.0)--(10.6,0.0);
\begin{scope}[shift={(11.440000000000001,0.0)}]
\coordinate (p0n5v1) at (90.0:.62);
\node[font=\tiny,inner sep=0pt] at (90.0:.79){$1$};
\coordinate (p0n5v2) at (38.57142857142857:.62);
\node[font=\tiny,inner sep=0pt] at (38.57142857142857:.79){$2$};
\coordinate (p0n5v3) at (-12.857142857142861:.62);
\node[font=\tiny,inner sep=0pt] at (-12.857142857142861:.79){$3$};
\coordinate (p0n5v4) at (-64.28571428571428:.62);
\node[font=\tiny,inner sep=0pt] at (-64.28571428571428:.79){$4$};
\coordinate (p0n5v5) at (-115.71428571428572:.62);
\node[font=\tiny,inner sep=0pt] at (-115.71428571428572:.79){$5$};
\coordinate (p0n5v6) at (-167.14285714285717:.62);
\node[font=\tiny,inner sep=0pt] at (-167.14285714285717:.79){$6$};
\coordinate (p0n5v7) at (-218.57142857142856:.62);
\node[font=\tiny,inner sep=0pt] at (-218.57142857142856:.79){$7$};
\draw[gray] (p0n5v1)--(p0n5v2)--(p0n5v3)--(p0n5v4)--(p0n5v5)--(p0n5v6)--(p0n5v7)--cycle;
\draw[routeorange,very thick] (p0n5v1)--(p0n5v3);
\draw[routeblue,thin] (p0n5v1)--(p0n5v4);
\draw[routeblue,thin] (p0n5v4)--(p0n5v6);
\draw[routeblue,thin] (p0n5v4)--(p0n5v7);
\node at (0,-1.12){$S_5$};
\end{scope}
\node[anchor=west,font=\small\bfseries] at (-.8,-2.05){Six flips};
\begin{scope}[shift={(0.0,-3.5)}]
\coordinate (p1n0v1) at (90.0:.62);
\node[font=\tiny,inner sep=0pt] at (90.0:.79){$1$};
\coordinate (p1n0v2) at (38.57142857142857:.62);
\node[font=\tiny,inner sep=0pt] at (38.57142857142857:.79){$2$};
\coordinate (p1n0v3) at (-12.857142857142861:.62);
\node[font=\tiny,inner sep=0pt] at (-12.857142857142861:.79){$3$};
\coordinate (p1n0v4) at (-64.28571428571428:.62);
\node[font=\tiny,inner sep=0pt] at (-64.28571428571428:.79){$4$};
\coordinate (p1n0v5) at (-115.71428571428572:.62);
\node[font=\tiny,inner sep=0pt] at (-115.71428571428572:.79){$5$};
\coordinate (p1n0v6) at (-167.14285714285717:.62);
\node[font=\tiny,inner sep=0pt] at (-167.14285714285717:.79){$6$};
\coordinate (p1n0v7) at (-218.57142857142856:.62);
\node[font=\tiny,inner sep=0pt] at (-218.57142857142856:.79){$7$};
\draw[gray] (p1n0v1)--(p1n0v2)--(p1n0v3)--(p1n0v4)--(p1n0v5)--(p1n0v6)--(p1n0v7)--cycle;
\draw[routeblue,thin] (p1n0v2)--(p1n0v5);
\draw[routeblue,thin] (p1n0v2)--(p1n0v6);
\draw[routeblue,thin] (p1n0v2)--(p1n0v7);
\draw[routeblue,thin] (p1n0v3)--(p1n0v5);
\node at (0,-1.12){$T_0$};
\end{scope}
\draw[-{Stealth},gray] (0.84,-3.5)--(1.2400000000000002,-3.5);
\begin{scope}[shift={(2.08,-3.5)}]
\coordinate (p1n1v1) at (90.0:.62);
\node[font=\tiny,inner sep=0pt] at (90.0:.79){$1$};
\coordinate (p1n1v2) at (38.57142857142857:.62);
\node[font=\tiny,inner sep=0pt] at (38.57142857142857:.79){$2$};
\coordinate (p1n1v3) at (-12.857142857142861:.62);
\node[font=\tiny,inner sep=0pt] at (-12.857142857142861:.79){$3$};
\coordinate (p1n1v4) at (-64.28571428571428:.62);
\node[font=\tiny,inner sep=0pt] at (-64.28571428571428:.79){$4$};
\coordinate (p1n1v5) at (-115.71428571428572:.62);
\node[font=\tiny,inner sep=0pt] at (-115.71428571428572:.79){$5$};
\coordinate (p1n1v6) at (-167.14285714285717:.62);
\node[font=\tiny,inner sep=0pt] at (-167.14285714285717:.79){$6$};
\coordinate (p1n1v7) at (-218.57142857142856:.62);
\node[font=\tiny,inner sep=0pt] at (-218.57142857142856:.79){$7$};
\draw[gray] (p1n1v1)--(p1n1v2)--(p1n1v3)--(p1n1v4)--(p1n1v5)--(p1n1v6)--(p1n1v7)--cycle;
\draw[routeblue,thin] (p1n1v2)--(p1n1v6);
\draw[routeblue,thin] (p1n1v2)--(p1n1v7);
\draw[routeblue,thin] (p1n1v3)--(p1n1v5);
\draw[routeorange,very thick] (p1n1v3)--(p1n1v6);
\node at (0,-1.12){$T_1$};
\end{scope}
\draw[-{Stealth},gray] (2.92,-3.5)--(3.3200000000000003,-3.5);
\begin{scope}[shift={(4.16,-3.5)}]
\coordinate (p1n2v1) at (90.0:.62);
\node[font=\tiny,inner sep=0pt] at (90.0:.79){$1$};
\coordinate (p1n2v2) at (38.57142857142857:.62);
\node[font=\tiny,inner sep=0pt] at (38.57142857142857:.79){$2$};
\coordinate (p1n2v3) at (-12.857142857142861:.62);
\node[font=\tiny,inner sep=0pt] at (-12.857142857142861:.79){$3$};
\coordinate (p1n2v4) at (-64.28571428571428:.62);
\node[font=\tiny,inner sep=0pt] at (-64.28571428571428:.79){$4$};
\coordinate (p1n2v5) at (-115.71428571428572:.62);
\node[font=\tiny,inner sep=0pt] at (-115.71428571428572:.79){$5$};
\coordinate (p1n2v6) at (-167.14285714285717:.62);
\node[font=\tiny,inner sep=0pt] at (-167.14285714285717:.79){$6$};
\coordinate (p1n2v7) at (-218.57142857142856:.62);
\node[font=\tiny,inner sep=0pt] at (-218.57142857142856:.79){$7$};
\draw[gray] (p1n2v1)--(p1n2v2)--(p1n2v3)--(p1n2v4)--(p1n2v5)--(p1n2v6)--(p1n2v7)--cycle;
\draw[routeblue,thin] (p1n2v2)--(p1n2v6);
\draw[routeblue,thin] (p1n2v2)--(p1n2v7);
\draw[routeblue,thin] (p1n2v3)--(p1n2v6);
\draw[routeorange,very thick] (p1n2v4)--(p1n2v6);
\node at (0,-1.12){$T_2$};
\end{scope}
\draw[-{Stealth},gray] (5.0,-3.5)--(5.4,-3.5);
\begin{scope}[shift={(6.24,-3.5)}]
\coordinate (p1n3v1) at (90.0:.62);
\node[font=\tiny,inner sep=0pt] at (90.0:.79){$1$};
\coordinate (p1n3v2) at (38.57142857142857:.62);
\node[font=\tiny,inner sep=0pt] at (38.57142857142857:.79){$2$};
\coordinate (p1n3v3) at (-12.857142857142861:.62);
\node[font=\tiny,inner sep=0pt] at (-12.857142857142861:.79){$3$};
\coordinate (p1n3v4) at (-64.28571428571428:.62);
\node[font=\tiny,inner sep=0pt] at (-64.28571428571428:.79){$4$};
\coordinate (p1n3v5) at (-115.71428571428572:.62);
\node[font=\tiny,inner sep=0pt] at (-115.71428571428572:.79){$5$};
\coordinate (p1n3v6) at (-167.14285714285717:.62);
\node[font=\tiny,inner sep=0pt] at (-167.14285714285717:.79){$6$};
\coordinate (p1n3v7) at (-218.57142857142856:.62);
\node[font=\tiny,inner sep=0pt] at (-218.57142857142856:.79){$7$};
\draw[gray] (p1n3v1)--(p1n3v2)--(p1n3v3)--(p1n3v4)--(p1n3v5)--(p1n3v6)--(p1n3v7)--cycle;
\draw[routeblue,thin] (p1n3v2)--(p1n3v7);
\draw[routeblue,thin] (p1n3v3)--(p1n3v6);
\draw[routeorange,very thick] (p1n3v3)--(p1n3v7);
\draw[routeblue,thin] (p1n3v4)--(p1n3v6);
\node at (0,-1.12){$T_3$};
\end{scope}
\draw[-{Stealth},gray] (7.08,-3.5)--(7.48,-3.5);
\begin{scope}[shift={(8.32,-3.5)}]
\coordinate (p1n4v1) at (90.0:.62);
\node[font=\tiny,inner sep=0pt] at (90.0:.79){$1$};
\coordinate (p1n4v2) at (38.57142857142857:.62);
\node[font=\tiny,inner sep=0pt] at (38.57142857142857:.79){$2$};
\coordinate (p1n4v3) at (-12.857142857142861:.62);
\node[font=\tiny,inner sep=0pt] at (-12.857142857142861:.79){$3$};
\coordinate (p1n4v4) at (-64.28571428571428:.62);
\node[font=\tiny,inner sep=0pt] at (-64.28571428571428:.79){$4$};
\coordinate (p1n4v5) at (-115.71428571428572:.62);
\node[font=\tiny,inner sep=0pt] at (-115.71428571428572:.79){$5$};
\coordinate (p1n4v6) at (-167.14285714285717:.62);
\node[font=\tiny,inner sep=0pt] at (-167.14285714285717:.79){$6$};
\coordinate (p1n4v7) at (-218.57142857142856:.62);
\node[font=\tiny,inner sep=0pt] at (-218.57142857142856:.79){$7$};
\draw[gray] (p1n4v1)--(p1n4v2)--(p1n4v3)--(p1n4v4)--(p1n4v5)--(p1n4v6)--(p1n4v7)--cycle;
\draw[routeblue,thin] (p1n4v2)--(p1n4v7);
\draw[routeblue,thin] (p1n4v3)--(p1n4v7);
\draw[routeblue,thin] (p1n4v4)--(p1n4v6);
\draw[routeorange,very thick] (p1n4v4)--(p1n4v7);
\node at (0,-1.12){$T_4$};
\end{scope}
\draw[-{Stealth},gray] (9.16,-3.5)--(9.56,-3.5);
\begin{scope}[shift={(10.4,-3.5)}]
\coordinate (p1n5v1) at (90.0:.62);
\node[font=\tiny,inner sep=0pt] at (90.0:.79){$1$};
\coordinate (p1n5v2) at (38.57142857142857:.62);
\node[font=\tiny,inner sep=0pt] at (38.57142857142857:.79){$2$};
\coordinate (p1n5v3) at (-12.857142857142861:.62);
\node[font=\tiny,inner sep=0pt] at (-12.857142857142861:.79){$3$};
\coordinate (p1n5v4) at (-64.28571428571428:.62);
\node[font=\tiny,inner sep=0pt] at (-64.28571428571428:.79){$4$};
\coordinate (p1n5v5) at (-115.71428571428572:.62);
\node[font=\tiny,inner sep=0pt] at (-115.71428571428572:.79){$5$};
\coordinate (p1n5v6) at (-167.14285714285717:.62);
\node[font=\tiny,inner sep=0pt] at (-167.14285714285717:.79){$6$};
\coordinate (p1n5v7) at (-218.57142857142856:.62);
\node[font=\tiny,inner sep=0pt] at (-218.57142857142856:.79){$7$};
\draw[gray] (p1n5v1)--(p1n5v2)--(p1n5v3)--(p1n5v4)--(p1n5v5)--(p1n5v6)--(p1n5v7)--cycle;
\draw[routeorange,very thick] (p1n5v1)--(p1n5v3);
\draw[routeblue,thin] (p1n5v3)--(p1n5v7);
\draw[routeblue,thin] (p1n5v4)--(p1n5v6);
\draw[routeblue,thin] (p1n5v4)--(p1n5v7);
\node at (0,-1.12){$T_5$};
\end{scope}
\draw[-{Stealth},gray] (11.24,-3.5)--(11.64,-3.5);
\begin{scope}[shift={(12.48,-3.5)}]
\coordinate (p1n6v1) at (90.0:.62);
\node[font=\tiny,inner sep=0pt] at (90.0:.79){$1$};
\coordinate (p1n6v2) at (38.57142857142857:.62);
\node[font=\tiny,inner sep=0pt] at (38.57142857142857:.79){$2$};
\coordinate (p1n6v3) at (-12.857142857142861:.62);
\node[font=\tiny,inner sep=0pt] at (-12.857142857142861:.79){$3$};
\coordinate (p1n6v4) at (-64.28571428571428:.62);
\node[font=\tiny,inner sep=0pt] at (-64.28571428571428:.79){$4$};
\coordinate (p1n6v5) at (-115.71428571428572:.62);
\node[font=\tiny,inner sep=0pt] at (-115.71428571428572:.79){$5$};
\coordinate (p1n6v6) at (-167.14285714285717:.62);
\node[font=\tiny,inner sep=0pt] at (-167.14285714285717:.79){$6$};
\coordinate (p1n6v7) at (-218.57142857142856:.62);
\node[font=\tiny,inner sep=0pt] at (-218.57142857142856:.79){$7$};
\draw[gray] (p1n6v1)--(p1n6v2)--(p1n6v3)--(p1n6v4)--(p1n6v5)--(p1n6v6)--(p1n6v7)--cycle;
\draw[routeblue,thin] (p1n6v1)--(p1n6v3);
\draw[routeorange,very thick] (p1n6v1)--(p1n6v4);
\draw[routeblue,thin] (p1n6v4)--(p1n6v6);
\draw[routeblue,thin] (p1n6v4)--(p1n6v7);
\node at (0,-1.12){$T_6$};
\end{scope}
\end{tikzpicture}
\caption{Every intermediate triangulation on the five-flip path (top) and six-flip path (bottom). The orange thick diagonal is the one inserted at that step; all other current diagonals are blue. Here $S_0=T_0=\mathcal S$ and $S_5=T_6=\mathcal U$. The six-flip route has lower cost at $L=100$, despite having one extra step.}\label{fig:paths}
\end{figure}
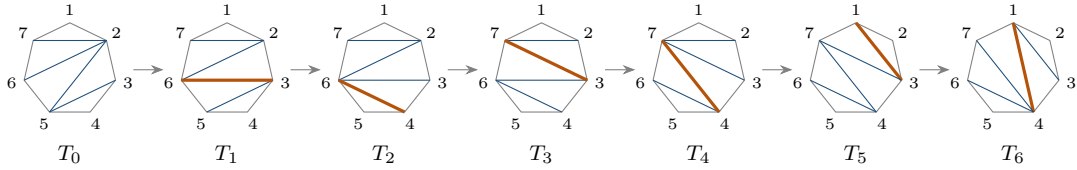

The costs are approximately
\begin{equation*}
W_5=1.591447397118036,\qquad W_6=1.316916488633198.
\end{equation*}
The underlying exact comparison is
\begin{equation}\label{eq:exactratio}
1.3159<\frac{P_6(\psi)}{P_5(\psi)}<1.3160.
\end{equation}
For clarity, this strict conclusion is certified with rational arithmetic, not inferred from the displayed decimals. For each positive integer $N$ occurring under a cube root, put $s=10^{30}$ and find the integer $h$ with $h^3\le Ns^3<(h+1)^3$. Then
\begin{equation*}
\frac hs\le\sqrt[3]{N}<\frac{h+1}s.
\end{equation*}
Applying these enclosures to each factor in~\eqref{eq:g}, multiplying positive lower and upper bounds, and dividing the resulting path products proves~\eqref{eq:exactratio}. The factors appear in table~\ref{tab:pairs} at $L=100$; the exact computation is included in the source package.

By Lemma~\ref{lem:unique}, the six-flip path outperforms every geodesic for every input. Together with Lemma~\ref{lem:finite}, this completes the proof of Theorem~\ref{thm:minimal}.

\section{A uniform one-parameter separation}\label{sec:family}
Keep the endpoints and paths, and set
\begin{equation*}
t(L)=(0,1,2,3,L,L+1,2L),\qquad L>3.
\end{equation*}
Let $W_5(L)$ and $W_6(L)$ be their costs. Define
\begin{equation*}
R(L)=\exp\bigl(W_5(L)-W_6(L)\bigr).
\end{equation*}
The successive arguments of $g$ are listed in table~\ref{tab:pairs}.

\begin{table}[tb]
\centering
\caption{Positive pairs $(A,B)$ for the two seven-gon paths. Homogeneity permits common factors to be cancelled.}\label{tab:pairs}
\small
\begin{tabular}{cll}\toprule
Step & Five flips & Six flips\\\midrule
1 & $(L-3,L-1)$ & $(1,L(L-2))$\\
2 & $(2,L(L-3))$ & $(1,(L-1)(L-3))$\\
3 & $(2(L-1),(2L-1)(L-2))$ & $(L-1,(2L-1)(L-1))$\\
4 & $(2L-3,4L)$ & $(L-1,(2L-2)(L-2))$\\
5 & $(1,3)$ & $(2L-2,2L)$\\
6 & --- & $(4L-6,2L)$\\\bottomrule
\end{tabular}
\end{table}

\begin{Proposition}\label{prop:family}
The ratio satisfies
\begin{equation*}
\lim_{L\to\infty}R(L)=C,\qquad R(L)>e^{2/15}\quad(L\ge1000),
\end{equation*}
where $C$ is defined in~\eqref{eq:C}.
\end{Proposition}
\begin{proof}
Put $f(r)=\log g(1,r)$ for $r>0$. Then $f(r)=f(r^{-1})$, and
\begin{equation*}
\frac{df}{d\log r}=\frac13\left(\frac{r^{1/3}}{1+r^{1/3}}-\frac r{1+r}\right).
\end{equation*}
In particular,
\begin{equation}\label{eq:fbound}
\left|\frac{df}{d\log r}\right|\le\frac13,\qquad
0<f(r)\le r^{1/3},\qquad f(r)\le f(1).
\end{equation}
The middle inequality follows from $\log(1+r^{1/3})\le r^{1/3}$; the last follows from the derivative's sign on either side of one.

Taking limits in table~\ref{tab:pairs}, the five-flip product tends to $g(1,1)g(1,2)g(1,3)$, while the six-flip product tends to $g(1,1)g(1,2)$. This proves the stated limit.

For a uniform estimate, keep the constant contribution $f(3)$ from the five-flip path and discard its second and third positive contributions. Bound its first contribution below by
$f(1)-\frac13\log((L-1)/(L-3))$. Bound the fifth contribution of the six-flip path above by $f(1)$. The difference between the fourth contribution of the five-flip path and the sixth contribution of the six-flip path is bounded below by
$-\frac23\log(2L/(2L-3))$, using~\eqref{eq:fbound} and symmetry. Bound the remaining four six-flip contributions by their small-argument cube roots. Therefore
\begin{equation*}
\log R(L)\ge\log C-E(L),
\end{equation*}
where
\begin{align*}
E(L)={}&\frac13\log\frac{L-1}{L-3}+\frac23\log\frac{2L}{2L-3}\notag\\
&+[L(L-2)]^{-1/3}+[(L-1)(L-3)]^{-1/3}\notag\\
&+(2L-1)^{-1/3}+[2(L-2)]^{-1/3}.
\end{align*}
For $L\ge1000$, the two logarithmic terms together are less than $1/100$, and the remaining terms are bounded by $1/90$, $1/90$, $1/12$, $1/12$, respectively. Their sum is less than $1/5$. Finally $\sqrt[3]{3}>7/5$ and $\sqrt[3]{4}<8/5$, so $C>3/2$, and
$\log(3/2)=\int_1^{3/2}x^{-1}\,dx>1/3$. Hence $\log R(L)>2/15$.
\end{proof}

This proves a separation for every real $L\ge1000$. It does not assert global monotonicity of $R$ or global optimality of the displayed six-flip path at every $L$.

Figure~\ref{fig:familycurve} illustrates the ratio and its limiting value; the proof of the uniform bound is analytic.
\begin{figure}[tbp]
\centering
\begin{tikzpicture}
\begin{semilogxaxis}[width=.86\textwidth,height=6.5cm,xlabel={$L$},ylabel={$R(L)=P_6/P_5$},xmin=3,xmax=1e6,ymin=.4,ymax=1.62,legend style={at={(.98,.03)},anchor=south east}]
\addplot[routeblue,very thick] coordinates {(3.01,0.47027868816) (3.276242594,0.604445287982) (3.566035062,0.662699441796) (3.881460453,0.707282969252) (4.224786069,0.745134356924) (4.598479758,0.77875893821) (5.005227658,0.809372714206) (5.447953505,0.837668301407) (5.929839643,0.864083484973) (6.454349906,0.888916319063) (7.025254512,0.9123814157) (7.646657166,0.934640216647) (8.323024556,0.955818498371) (9.05921846,0.976017094919) (9.860530695,0.995318783256) (10.73272115,1.01379288187) (11.6820592,1.0314984264) (12.71536875,1.04848642634) (13.84007731,1.06480150914) (15.06426937,1.08048314424) (16.3967445,1.09556657189) (17.84708065,1.11008352) (19.42570293,1.12406276559) (21.14395859,1.13753058052) (23.01419859,1.15051108959) (25.04986634,1.16302656126) (27.26559439,1.17509764601) (29.67730953,1.1867435735) (32.30234737,1.19798231698) (35.15957686,1.20883073137) (38.26953599,1.21930467003) (41.65457936,1.22941908419) (45.33903892,1.23918810808) (49.34939884,1.24862513228) (53.71448587,1.25774286719) (58.4656766,1.26655339838) (63.63712292,1.2750682349) (69.26599757,1.28329835193) (75.39276133,1.29125422832) (82.06145382,1.29894588004) (89.32001009,1.30638288995) (97.22060518,1.31357443455) (105.8200291,1.32052930802) (115.1800953,1.32725594399) (125.3680845,1.33376243529) (136.4572287,1.34005655203) (148.5272375,1.34614575808) (161.6648711,1.35203722628) (175.9645637,1.35773785247) (191.5291026,1.36325426849) (208.4703667,1.36859285428) (226.9101313,1.37375974912) (246.9809426,1.37876086224) (268.8270712,1.38360188264) (292.6055485,1.38828828847) (318.4872962,1.39282535578) (346.658354,1.39721816686) (377.3212176,1.40147161804) (410.696294,1.40559042721) (447.0234857,1.4095791409) (486.5639151,1.41344214101) (529.6018019,1.41718365129) (576.4465055,1.42080774343) (627.4347491,1.42431834297) (682.9330399,1.42771923491) (743.3403038,1.43101406906) (809.0907527,1.4342063652) (880.6570056,1.43729951802) (958.553486,1.44029680182) (1043.34012,1.44320137509) (1135.626359,1.44601628483) (1236.075564,1.44874447073) (1345.409772,1.45138876923) (1464.414884,1.45395191731) (1593.946319,1.45643655625) (1734.935157,1.45884523517) (1888.394838,1.46118041444) (2055.428439,1.46344446902) (2237.236612,1.46563969155) (2435.126207,1.46776829546) (2650.519668,1.46983241786) (2884.965261,1.47183412234) (3140.148196,1.47377540171) (3417.902749,1.47565818059) (3720.225438,1.47748431787) (4049.289383,1.47925560918) (4407.459918,1.48097378914) (4797.3116,1.48264053363) (5221.646712,1.48425746191) (5683.515406,1.48582613868) (6186.237628,1.48734807603) (6733.426982,1.4888247354) (7329.016706,1.49025752936) (7977.287943,1.4916478234) (8682.900514,1.49299693758) (9450.926414,1.49430614822) (10286.88627,1.49557668942) (11196.78902,1.49680975457) (12187.17512,1.49800649781) (13265.16354,1.49916803544) (14438.50293,1.50029544721) (15715.62734,1.50138977767) (17105.71684,1.50245203736) (18618.76349,1.50348320403) (20265.64319,1.5044842238) (22058.19383,1.50545601224) (24009.3004,1.50639945542) (26132.98758,1.507315411) (28444.52061,1.50820470915) (30960.51495,1.50906815352) (33699.05576,1.50990652215) (36679.82786,1.51072056837) (39924.25727,1.51151102162) (43455.66519,1.51227858827) (47299.43564,1.51302395242) (51483.19792,1.51374777664) (56037.02523,1.51445070269) (60993.6508,1.51513335227) (66388.70324,1.51579632761) (72260.96257,1.51644021222) (78652.63904,1.51706557144) (85609.67647,1.51767295306) (93182.08258,1.51826288796) (101424.2884,1.51883589057) (110395.5394,1.51939245951) (120160.3218,1.51993307806) (130788.8254,1.52045821466) (142357.4487,1.5209683234) (154949.3478,1.52146384454) (168655.0343,1.52194520486) (183573.0257,1.52241281818) (199810.5535,1.52286708577) (217484.3342,1.5233083967) (236721.4085,1.5237371283) (257660.0536,1.52415364648) (280450.7783,1.52455830615) (305257.4038,1.52495145152) (332258.2421,1.52533341647) (361647.3771,1.52570452487) (393636.0601,1.5260650909) (428454.2282,1.52641541935) (466352.1569,1.52675580592) (507602.2593,1.52708653751) (552501.0442,1.52740789247) (601371.2473,1.5277201409) (654564.1513,1.52802354491) (712462.1108,1.52831835882) (775481.3004,1.52860482945) (844074.7069,1.52888319635) (918735.3846,1.52915369199) (1000000,1.52941654201)};\addlegendentry{explicit path ratio}
\addplot[black,dashed,domain=3:1e6] {1.5385208213635064};\addlegendentry{limit $C$}
\addplot[gray,densely dotted,domain=3:1e6] {1};\addlegendentry{equal success}
\addplot[routeorange,only marks,mark=*,mark size=2.5pt] coordinates {(100,1.3159132463251584)};
\node[anchor=north west,font=\footnotesize] at (axis cs:100,1.3159132463251584){$L=100$};
\end{semilogxaxis}
\end{tikzpicture}
\caption{Numerical evaluation of the explicit ratio from table~\ref{tab:pairs}, with a logarithmic horizontal axis. At $L=100$, the rigorous enclosure is $1.3159<R(100)<1.3160$. Proposition~\ref{prop:family} proves $R(L)>e^{2/15}$ for every real $L\ge1000$ and $R(L)\to C$; the plot does not serve as a proof of monotonicity.}\label{fig:familycurve}
\end{figure}
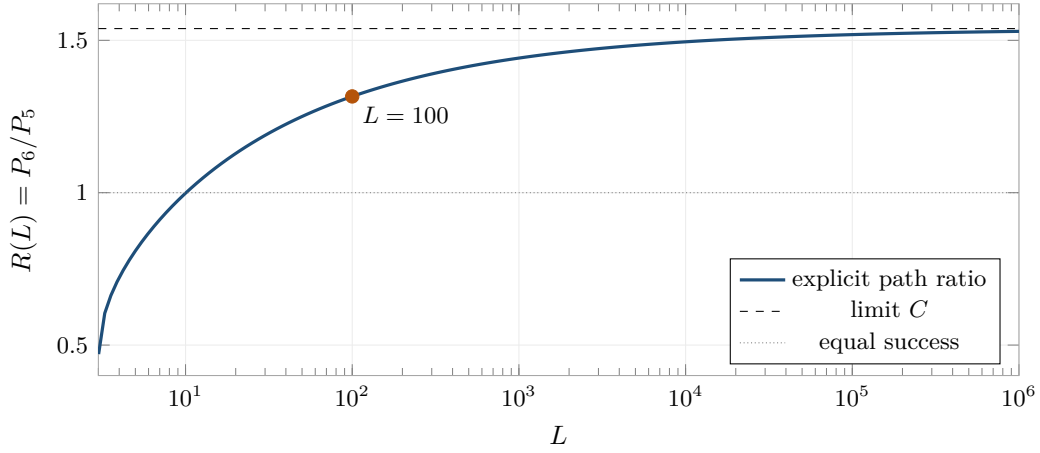

\section{Exponential separation in connected type}\label{sec:gluing}
\subsection{One polygon with several seven-gon regions}
For $k\ge1$, take a polygon with vertices $v_0,\ldots,v_{5k+1}$. The common diagonals
\begin{equation*}
\delta_j=(v_0,v_{5j+1}),\qquad1\le j\le k-1,
\end{equation*}
divide it into the seven-gons
\begin{equation*}
H_j=(v_0,v_{5j+1},\ldots,v_{5j+6}),\qquad0\le j<k.
\end{equation*}
Insert the local triangulations $\mathcal S$ and $\mathcal U$ from~\eqref{eq:endpoints} in every region, together with all common diagonals. The ambient exchange type is connected $A_{5k-1}$; the common diagonals are mutable, not frozen. Only the face obtained by retaining them is a product of seven-gon flip graphs.

The two-block case is shown in figure~\ref{fig:gluing}; all internal diagonals remain part of the connected exchange graph.
\begin{figure}[tbp]
\centering
\begin{tikzpicture}[font=\scriptsize]
\begin{scope}[xshift=0.0cm]
\coordinate (g0v0) at (90.0:2.3);
\node at (90.0:2.58){$v_{0}$};
\coordinate (g0v1) at (60.0:2.3);
\node at (60.0:2.58){$v_{1}$};
\coordinate (g0v2) at (30.0:2.3);
\node at (30.0:2.58){$v_{2}$};
\coordinate (g0v3) at (0.0:2.3);
\node at (0.0:2.58){$v_{3}$};
\coordinate (g0v4) at (-30.0:2.3);
\node at (-30.0:2.58){$v_{4}$};
\coordinate (g0v5) at (-60.0:2.3);
\node at (-60.0:2.58){$v_{5}$};
\coordinate (g0v6) at (-90.0:2.3);
\node at (-90.0:2.58){$v_{6}$};
\coordinate (g0v7) at (-120.0:2.3);
\node at (-120.0:2.58){$v_{7}$};
\coordinate (g0v8) at (-150.0:2.3);
\node at (-150.0:2.58){$v_{8}$};
\coordinate (g0v9) at (-180.0:2.3);
\node at (-180.0:2.58){$v_{9}$};
\coordinate (g0v10) at (-210.0:2.3);
\node at (-210.0:2.58){$v_{10}$};
\coordinate (g0v11) at (-240.0:2.3);
\node at (-240.0:2.58){$v_{11}$};
\fill[routeblue!6] (g0v0)--(g0v1)--(g0v2)--(g0v3)--(g0v4)--(g0v5)--(g0v6)--cycle;
\fill[routeorange!7] (g0v0)--(g0v6)--(g0v7)--(g0v8)--(g0v9)--(g0v10)--(g0v11)--cycle;
\draw (g0v0)--(g0v1)--(g0v2)--(g0v3)--(g0v4)--(g0v5)--(g0v6)--(g0v7)--(g0v8)--(g0v9)--(g0v10)--(g0v11)--cycle;
\draw[black,very thick,dashed] (g0v0)--(g0v6);
\draw[routeblue,thick] (g0v1)--(g0v4);
\draw[routeblue,thick] (g0v1)--(g0v5);
\draw[routeblue,thick] (g0v1)--(g0v6);
\draw[routeblue,thick] (g0v2)--(g0v4);
\draw[routeblue,thick] (g0v6)--(g0v9);
\draw[routeblue,thick] (g0v6)--(g0v10);
\draw[routeblue,thick] (g0v6)--(g0v11);
\draw[routeblue,thick] (g0v7)--(g0v9);
\node[fill=white,inner sep=2pt] at (0,0){$\delta_1$};
\node[font=\small] at (0,-3){$\mathcal S_2$};
\end{scope}
\begin{scope}[xshift=7.1cm]
\coordinate (g1v0) at (90.0:2.3);
\node at (90.0:2.58){$v_{0}$};
\coordinate (g1v1) at (60.0:2.3);
\node at (60.0:2.58){$v_{1}$};
\coordinate (g1v2) at (30.0:2.3);
\node at (30.0:2.58){$v_{2}$};
\coordinate (g1v3) at (0.0:2.3);
\node at (0.0:2.58){$v_{3}$};
\coordinate (g1v4) at (-30.0:2.3);
\node at (-30.0:2.58){$v_{4}$};
\coordinate (g1v5) at (-60.0:2.3);
\node at (-60.0:2.58){$v_{5}$};
\coordinate (g1v6) at (-90.0:2.3);
\node at (-90.0:2.58){$v_{6}$};
\coordinate (g1v7) at (-120.0:2.3);
\node at (-120.0:2.58){$v_{7}$};
\coordinate (g1v8) at (-150.0:2.3);
\node at (-150.0:2.58){$v_{8}$};
\coordinate (g1v9) at (-180.0:2.3);
\node at (-180.0:2.58){$v_{9}$};
\coordinate (g1v10) at (-210.0:2.3);
\node at (-210.0:2.58){$v_{10}$};
\coordinate (g1v11) at (-240.0:2.3);
\node at (-240.0:2.58){$v_{11}$};
\fill[routeblue!6] (g1v0)--(g1v1)--(g1v2)--(g1v3)--(g1v4)--(g1v5)--(g1v6)--cycle;
\fill[routeorange!7] (g1v0)--(g1v6)--(g1v7)--(g1v8)--(g1v9)--(g1v10)--(g1v11)--cycle;
\draw (g1v0)--(g1v1)--(g1v2)--(g1v3)--(g1v4)--(g1v5)--(g1v6)--(g1v7)--(g1v8)--(g1v9)--(g1v10)--(g1v11)--cycle;
\draw[routeblue,thick] (g1v0)--(g1v2);
\draw[routeblue,thick] (g1v0)--(g1v3);
\draw[black,very thick,dashed] (g1v0)--(g1v6);
\draw[routeblue,thick] (g1v0)--(g1v7);
\draw[routeblue,thick] (g1v0)--(g1v8);
\draw[routeblue,thick] (g1v3)--(g1v5);
\draw[routeblue,thick] (g1v3)--(g1v6);
\draw[routeblue,thick] (g1v8)--(g1v10);
\draw[routeblue,thick] (g1v8)--(g1v11);
\node[fill=white,inner sep=2pt] at (0,0){$\delta_1$};
\node[font=\small] at (0,-3){$\mathcal U_2$};
\end{scope}
\end{tikzpicture}
\caption{Connected gluing with two seven-gon regions in a twelve-gon, of full type $A_9$. The common dashed diagonal $\delta_1=(v_0,v_6)$ separates the two blocks and remains mutable. All unweighted geodesics retain it, giving ten flips; the blockwise competitor uses twelve flips. The shaded regions receive the compatible coordinates in~\eqref{eq:globalcoords}.}\label{fig:gluing}
\end{figure}
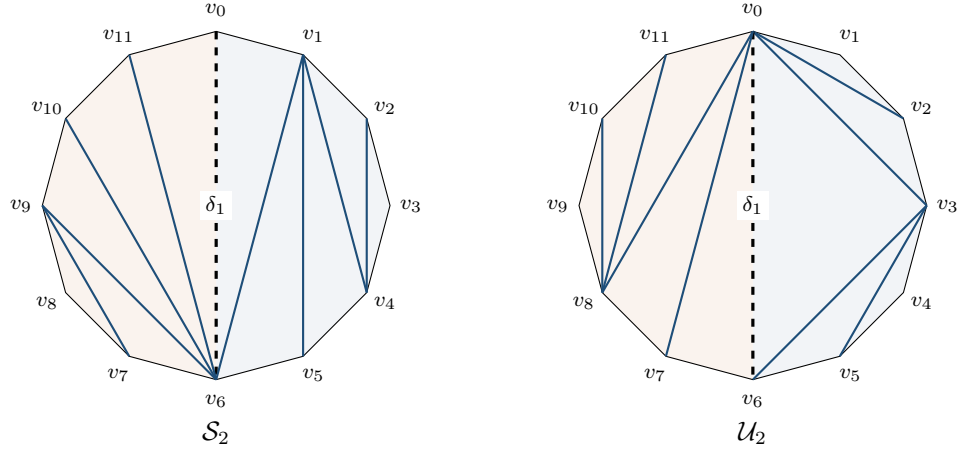

\subsection{Compatible global positive data}
Local coefficient assignments must come from one global positive configuration. Define the projective map
\begin{equation*}
\phi_L(x)=\frac{2L}{2L-1}\left(1-\frac1x\right).
\end{equation*}
It sends $0,1,2L$ to $\infty,0,1$. In region $H_j$, assign coordinates
\begin{equation}\label{eq:globalcoords}
\infty,\ j,\ j+\phi_L(2),\ j+\phi_L(3),\ j+\phi_L(L),\ j+\phi_L(L+1),\ j+1.
\end{equation}
Neighbouring regions agree on shared vertices, and all finite coordinates increase strictly. Choose $u_\infty=(0,-1)$ and $u_x=(1,x)$. Every ordered determinant is positive. The ratio $A/B$ is invariant under a common projective transformation and under rescaling each homogeneous vector, because each vertex factor occurs once in numerator and denominator. Hence every regional edge cost agrees exactly with the original seven-gon cost.

\subsection{Geodesics and the success bound}
The type-$A$ non-leaving-face property states that every flip geodesic retains all diagonals common to its endpoints \cite{STT,Ceballos}; see also the uniform Coxeter-theoretic treatment in \cite{Williams}. Thus every geodesic between the constructed endpoints remains in the product face. By Lemma~\ref{lem:unique}, its regional projections are the unique five-flip paths. Every global geodesic consequently has length $5k$ and cost $kW_5(L)$.

Interleaving the regional six-flip paths gives a feasible path with length $6k$ and cost $kW_6(L)$. Proposition~\ref{prop:path} yields
\begin{equation}\label{eq:amplify}
\frac{P_{\opt}(\psi)}{P_{\geo}(\psi)}\ge R(L)^k.
\end{equation}
For a general $n\ge4$, write $n+1=5k+r$ with $0\le r\le4$. Add $r$ vertices beyond the last finite coordinate and retain the additional fan diagonals in both endpoint triangulations. These cut off $r$ common triangles and do not change the original block data. The full polygon has $n+3$ vertices and connected type $A_n$. The same non-leaving-face argument gives~\eqref{eq:amplify}.

Taking $L\ge1000$ and applying Proposition~\ref{prop:family} proves~\eqref{eq:mainbound}. For $1<c<C$, choose one $L$ with $R(L)>c$; that choice works for every $n$. This completes the proof of Theorem~\ref{thm:exponential}.

\begin{Remark}
Weighted-optimal paths may leave the common face. This causes no difficulty: any such improvement only increases the left-hand side of~\eqref{eq:amplify}. We constrain unweighted geodesics using the cited theorem, and provide a feasible weighted competitor. We do not infer a weighted non-leaving-face theorem.
\end{Remark}

Figure~\ref{fig:amplification} shows the quantitative amplification at the fixed admissible parameter $L=1000$.
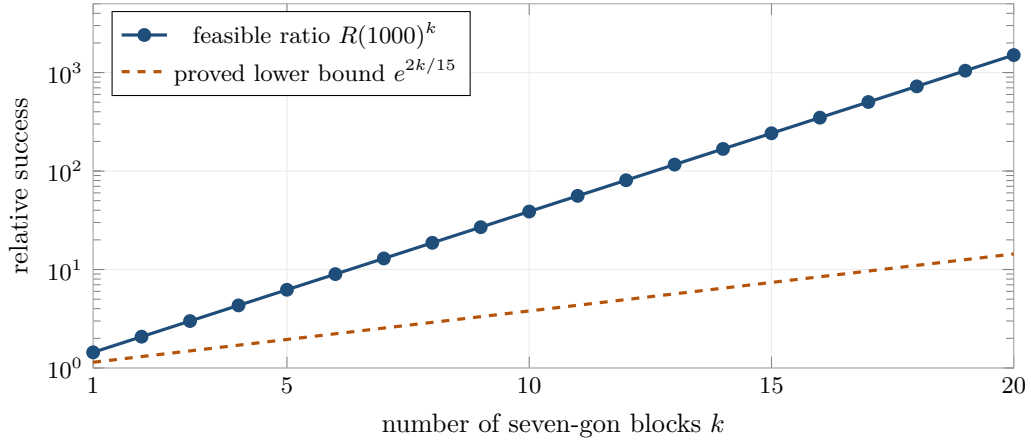
\begin{figure}[tbp]
\centering
\begin{tikzpicture}
\begin{semilogyaxis}[width=.86\textwidth,height=6.4cm,xlabel={number of seven-gon blocks $k$},ylabel={relative success},xmin=1,xmax=20,xtick={1,5,10,15,20},ymin=1,ymax=5000,legend style={at={(.02,.98)},anchor=north west}]
\addplot[routeblue,very thick,domain=1:20,samples=20,mark=*] {exp(x*0.36586376901779627)};\addlegendentry{feasible ratio $R(1000)^k$}
\addplot[routeorange,dashed,very thick,domain=1:20,samples=20] {exp(2*x/15)};\addlegendentry{proved lower bound $e^{2k/15}$}
\end{semilogyaxis}
\end{tikzpicture}
\caption{Amplification in one connected polygon. At rank $n=5k-1$ (and after adding up to four common ears), the blue curve is the explicit blockwise competitor's success divided by geodesic success. The global optimum is at least this large. The dashed curve is the uniform analytic lower bound. The logarithmic vertical axis displays exponential growth; neither curve is an absolute success probability.}\label{fig:amplification}
\end{figure}

\subsection{Stability under compatible coefficient perturbations}
The strict separation is not confined to a single positive configuration. The derivative bound in~\eqref{eq:fbound} gives an explicit tolerance in logarithmic coefficient coordinates.

\begin{Corollary}[Stable exponential separation]\label{cor:stability}
Fix $L\ge1000$ in the connected construction and write
$\delta(L)=W_5(L)-W_6(L)>2/15$. Perturb the global positive Pl\"ucker data while keeping the endpoint triangulations fixed. Suppose that, on every flip of each of the prescribed regional five- and six-flip paths, the coefficient $Y=A/B$ satisfies
\begin{equation*}
|\log \widetilde Y-\log Y|\le\varepsilon.
\end{equation*}
Then the perturbed sequential probabilities obey
\begin{equation}\label{eq:robust}
\frac{\widetilde P_{\opt}(\psi)}{\widetilde P_{\geo}(\psi)}
\ge \exp\!\left[k\left(\delta(L)-\frac{11\varepsilon}{3}\right)\right].
\end{equation}
In particular, $0\le\varepsilon<2/55$ preserves an exponential separation uniformly in the number of blocks.
\end{Corollary}
\begin{proof}
Since $w=f(Y)$, the mean value theorem and~\eqref{eq:fbound} bound the change of each edge weight by $\varepsilon/3$. The difference between a regional five-edge cost and six-edge cost can therefore decrease by at most $11\varepsilon/3$. The underlying unweighted graph and its geodesics are unchanged. All geodesics still project to the unique five-edge route in every block, whereas the six-edge routes remain feasible. Summing the regional estimates and applying Proposition~\ref{prop:path} proves~\eqref{eq:robust}. Finally, $\delta(L)>2/15$ gives a positive exponent whenever $\varepsilon<2/55$.
\end{proof}
The perturbations must come from global positive Pl\"ucker data; the statement does not permit independent changes that violate coefficient compatibility. For any fixed rank, continuity of the finitely many positive cross-ratios gives a neighbourhood of the constructed configuration satisfying the hypothesis. The bound concerns coefficient uncertainty within the exact model, rather than errors in implemented Kraus operators.

\section{Discussion}
The quantity $W_\gamma$ measures accumulated normalization loss in an exact sequential implementation. Its difference between two paths determines a state-independent success ratio even though the absolute probabilities depend on the input. The exponential theorem is therefore a relative efficiency statement: it does not imply a lower bound independent of rank on the absolute success probability.

The conflict-reduction counterexamples of \cite{ClearyMaio} concern local strategies for the unweighted graph metric. In the present weighted setting, the strict pentagon inequality shows that elementary algebraic simplifications are useful local improvement rules. The seven-gon example shows their limitation: a globally better route can contain more flips. The low-rank certificate isolates where this discrepancy first becomes possible, while projective gluing proves its amplification within a single connected exchange type. Corollary~\ref{cor:stability} shows that a controlled perturbation of the compatible coefficients preserves this amplification.

Several questions remain. One is whether the low-rank reduction lemma admits a short conceptual proof replacing its finite certificate. Another is to determine conditions for weighted geodesics to remain in common faces. Higher cyclic dimensions require new singular-value and pentagon analyses; the three-level formula alone does not establish them. Finally, a physical implementation model that also counts deterministic operations or permits recovery after failure may lead to a different optimization objective.

\appendix
\section{Finite certificates and exact verification}\label{app:certificate}
Lemma~\ref{lem:finite} is checked by two independent constructions of the same finite graph. One generates triangulations recursively and identifies square and pentagonal faces by short cycles; the other enumerates noncrossing diagonal subsets and constructs faces by fixing common diagonals. Both enumerate every simple path and reproduce Table~\ref{tab:finite}. For each nongeodesic path, the verification records a next-step reduction together with a strictly decreasing termination rank.

The second construction proceeds as follows.
\begin{enumerate}[leftmargin=2em]
\item Generate every subset of $m-3$ polygon diagonals. Retain precisely those with no crossing pair. Join states differing in one diagonal.
\item For each fixed set of $m-5$ diagonals, collect states containing it. Check that the resulting face is a cycle with four or five vertices.
\item Enumerate simple paths between each unordered pair of distinct states. Compute shortest distances independently by breadth-first search.
\item Process lengths increasingly. At a fixed length, build components under square replacements that remain simple. For every nongeodesic component, find an exit obtained by a pentagon three-to-two replacement or by a square replacement followed by loop deletion.
\item Check that the exit is shorter and has already been certified. Choose a spanning tree towards the exit and record its distance for each member of the component.
\item Check every recorded move against a geometric face. Verify adjacency, unchanged endpoints, and strict decrease of the lexicographic pair consisting of length and tree distance. Check that every terminal path attains the independently computed endpoint distance.
\end{enumerate}
A square replacement can have equal length; its tree distance must then decrease. At an exit, length strictly decreases. This is the termination argument and prevents a cycle of cost-preserving rewrites from being mistaken for a reduction.

Numerical matrix checks are useful diagnostics but do not replace the projective representation theorem used in Section~\ref{sec:operators}. Likewise, the all-rank separation is proved by the explicit gluing construction, not extrapolated from a finite set of ranks.

\section*{Ethical Statement/Conflict of Interest}
This article complies with all publishing ethics. There is no conflict of interest with any institution or individual.

\section*{Funding}
This work is supported in part by the National Natural Science Foundation of China (No.~12271085) and the Natural Science Foundation of Heilongjiang Province (No. PL2024A007).

\section*{Data availability statement}
No experimental or observational datasets were generated or analysed during the current study. The finite combinatorial certificates and exact-arithmetic verification programs supporting the computer-assisted results are supplied with the submission; the analytic arguments and parameters defining the examples are included in the article.

\pdfbookmark[1]{References}{ref}
\LastPageEnding
\end{document}